\documentclass[preprint,11pt]{elsarticle}

\usepackage[letterpaper,margin=1in]{geometry}
\usepackage{graphics,pictex} 
\usepackage{amsmath}
\usepackage{amsfonts}
\usepackage{amsthm}
\usepackage{amssymb}
\usepackage{bm}
\usepackage{multirow}
\usepackage{diagbox}
\usepackage{graphicx,rawfonts,latexsym,color}
\usepackage{}

\DeclareMathOperator{\ddiv}{div}
\newcommand{\Nedelec}{N\'{e}d\'{e}lec }

\usepackage{color}
\usepackage{ulem} 
\definecolor{cola}{rgb}{.0,.4,.9}

\newtheorem{theorem}{Theorem}[section]
\newtheorem{lemma}[theorem]{Lemma}

\theoremstyle{definition}
\newtheorem{definition}[theorem]{Definition}

\graphicspath{{REVISION/PicPNG/}{REVISION/PicPDF/}{REVISION/PicEPS/}}
\numberwithin{figure}{section}
\numberwithin{table}{section}

\begin{document}

\begin{frontmatter}



\title{New stabilized discretizations for poroelasticity and the Stokes' equations}


\author[UniZar]{C. Rodrigo}\ead{carmenr@unizar.es}
\address[UniZar]{Departamento de Matem\'{a}tica Aplicada, IUMA,
Universidad de Zaragoza, Zaragoza, Spain}
\author[Tufts]{X. Hu}\ead{xiaozhe.hu@tufts.edu}
\author[Tufts]{P. Ohm}\ead{peter.ohm@tufts.edu}
\address[Tufts]{Department of Mathematics, Tufts University, 
Medford, Massachusetts  02155, USA}
\author[Tufts]{J.H. Adler}\ead{james.adler@tufts.edu}
\author[CWI]{F.J. Gaspar}\ead{F.J.Gaspar@cwi.nl}
\address[CWI]{CWI, Centrum Wiskunde \& Informatica, Science Park 123, 1090 Amsterdam, The Netherlands}
\author[PennState]{L.T. Zikatanov}\ead{ludmil@psu.edu}
\address[PennState]{Department of Mathematics, Penn State, 
University Park, Pennsylvania, 16802, USA}

\begin{abstract}
  In this work, we consider the popular P1-RT0-P0
  discretization of the three-field formulation of Biot's
  consolidation problem. Since this finite-element formulation does not satisfy an inf-sup
  condition uniformly with respect to the physical parameters, several
  issues arise in numerical simulations.  For example,
  when the permeability is small with respect to the mesh size,
  volumetric locking may occur. Thus, we propose a stabilization technique that
  enriches the piecewise linear finite-element space of the
  displacement with the
  span of edge/face bubble functions. We show that for Biot's
  model this does give rise to discretizations that are uniformly stable with
  respect to the physical parameters. We also propose a perturbation of the
  bilinear form, which allows for local elimination of the bubble
  functions and provides a uniformly stable scheme with the same
  number of degrees of freedom as the classical P1-RT0-P0 approach. We prove
  optimal stability and error estimates for this
  discretization. Finally, we show that this scheme can also be
  successfully applied to Stokes' equations, yielding a discrete
  problem with optimal approximation properties and with minimum
  number of degrees of freedom (equivalent to a P1-P0
  discretization).  Numerical tests confirm the theory for both
  poroelastic and Stokes' test problems.
\end{abstract}


\end{frontmatter}


\section{Introduction}
\label{sec:intro}

The interaction between the deformation and fluid flow in a
fluid-saturated porous medium is the object of study in poroelasticity
theory. Such coupling has been modelled in the early
one-dimensional work of Terzaghi \cite{terzaghi}.  A more
general three-dimensional mathematical formulation was then established
by Maurice Biot in several pioneering publications (see
\cite{biot1} and \cite{biot2}).  Biot's models are widely used
nowadays in the modeling of many applications in different fields,
ranging from geomechanics and petroleum engineering, to biomechanics.
The existence and uniqueness of the solution for these problems have
been investigated by Showalter in \cite{showalter} and by Zenisek in
\cite{zenisek}.  
Regarding the numerical simulation of the poroelasticity equations,
there have been numerous contributions using finite-difference
schemes~\cite{Gaspar2003, Gaspar2006} and finite-volume methods (see
\cite{nordbotten_FVM} for recent developments).  Finite-element
methods, which are the subject of this work, have also been
considered  (see for
example the monograph by Lewis and Schrefler~\cite{LewisSchrefler} and
the references therein).

Stable finite-element schemes are constructed by either choosing
discrete spaces satisfying appropriate inf-sup (or LBB) conditions, or
applying suitable stabilization techniques to unstable finite-element
pairs. For the two-field (displacements-pressure) formulation of 
Biot's problem, the classical Taylor-Hood elements belong to the first
class \cite{MuradLoula92, MuradLoula94, MuradLoulaThome}, as well as
the MINI element \cite{RGHZ2016}. On the other hand, a stabilized discretization based on
linear finite elements for both displacements and pressure has been
recently analyzed in \cite{RGHZ2016}, and belongs to the second type. Regarding three-field
formulations (including the Darcy velocity), a stable finite-element method
based on non-conforming Crouzeix-Raviart finite elements for the
displacements, lowest order Raviart-Thomas-\Nedelec elements for the
Darcy velocity, and piecewise constants for the pressure,
was proposed in \cite{Hu2017}. For a four-field formulation of the
problem, which includes the stress tensor, the fluid flux, the solid
displacement, and the pore pressure as unknowns, a stable approach is
proposed in \cite{lee_four_field}. In that work, two mixed finite
elements, one for linear elasticity and one for mixed Poisson,
are coupled for the spatial discretization.

This paper focuses on the three-field formulation, which has received a lot of attention from the point
of view of novel discretizations \cite{johannes,phillips1, phillips2,
  phillips3}, as well as for the design of efficient solvers
\cite{Castelleto2016, Bause, Almani}.  Because of its application to
existing reservoir engineering simulators, one of the most frequently
considered schemes is a three-field formulation based on piecewise
linear elements for displacements, Raviart-Thomas-\Nedelec elements for
the fluid flux, and piecewise constants for the
pressure. This element, however, does not satisfy an inf-sup condition
uniformly with respect to the physical parameters of the
problem. Thus, we propose a stabilization of this popular element
which gives rise to uniform error bounds, keeping the same number of
degrees of freedom as in the original method.

A consequence of our analysis is that a new stable scheme for the
Stokes' equations is derived. The resulting method can be seen as a
perturbation of the well-known unstable pair based on piecewise linear
and piecewise constant elements for velocities and pressure,
respectively (P1-P0). This yields a stable finite-element
pair for Stokes, which has the lowest possible number of degrees of freedom.

The rest of the paper is organized as follows. Section \ref{sec:poro}
is devoted to describing Biot's problem and, in particular, the
considered three-field formulation and its discretization.  A
numerical example is given, illustrating
the difficulties that appear.
In Section \ref{sec:stab}, we introduce the stabilized
scheme in which we consider the enrichment of the piecewise linear
continuous finite-element space with edge/face (2D/3D) bubble
functions. Section \ref{sec:no_bubbles} is devoted to the local
elimination of the bubbles to maintain the same number of degrees of
freedom as in the original scheme. The well-posedness of the resulting
scheme, as well as the corresponding error analysis are also provided
here. In Section \ref{sec:stokes}, we present the
Stokes-stable finite-element method based on P1-P0 finite elements
obtained by following the same strategy as presented in the previous
sections for poroelasticity. Finally,
in Section \ref{sec:results}, we confirm the uniform convergence properties
of the stabilized schemes for both poroelasticity and Stokes' equations
through some numerical tests.


\section{Preliminaries: model problem and notation}
\label{sec:poro}

We consider the quasi-static Biot's model for 
consolidation in a linearly elastic, homogeneous, and 
isotropic porous medium saturated by an incompressible Newtonian fluid.  
According to Biot's theory~\cite{biot1}, 
the mathematical model of the consolidation process is described by 
the following system of partial differential equations (PDEs)
in a domain $\Omega\subset\mathbb{R}^d$, $d=2,3$ with sufficiently smooth 
boundary $\Gamma=\partial\Omega$:
\begin{eqnarray}
\mbox{\rm equilibrium equation:} & & -{\rm div} \, {\boldsymbol
\sigma}' +
\alpha \nabla \, p = \rho {\bm g}, \quad {\rm in} \, \Omega, \label{eq11} \\
\mbox{\rm constitutive equation:} & & \bm{\sigma}' = 2\mu {\boldsymbol \varepsilon}(\bm{u})  + \lambda\ddiv(\bm{u}) {\bm I}, \quad {\rm in} \, \Omega,
\label{eq12} \\
\mbox{\rm compatibility condition:} & & {\boldsymbol \varepsilon}({\bm u}) = \frac{1}{2}(\nabla {\bm u} + \nabla
{\bm u}^t), \quad {\rm in} \, \Omega,
\label{eq13} \\
\mbox{\rm Darcy's law:} & & {\bm w} = - \displaystyle \frac{1}{\mu_f}{\bm K} ( \nabla
p - \rho_f {\bm g}), \quad {\rm in} \, \Omega,
\label{eq14} \\
\mbox{\rm continuity equation:} & & \displaystyle \frac{\partial}{\partial t}\left( \frac{1}{M}p +
\alpha \ {\ddiv} \, {\bm u}\right)  + {\ddiv} \, {\bm w} = f, \quad {\rm in} \, \Omega,
\label{eq15}
\end{eqnarray}
where $\lambda$ and $\mu$ are the Lam\'{e} coefficients, $M$ is the
Biot modulus, and $\displaystyle \alpha = 1-\frac{K_b}{K_s}$ is the
Biot-Willis constant. Here, $K_b$ and $K_s$ denote the drained and the
solid phase bulk moduli. As is customary, ${\bm K}$ stands for the absolute
permeability tensor, $\mu_f$ is the viscosity of the fluid, and
$\boldsymbol I$ is the identity tensor.  The unknown functions are the
displacement vector ${\bm u}$ and the pore pressure $p$. 
The effective stress tensor and the strain tensor are
denoted by ${\boldsymbol \sigma}'$ and $\boldsymbol \varepsilon$,
respectively.  The percolation velocity of the fluid, or Darcy's
velocity, relative to the
soil is denoted by $\bm w$ and the vector-valued function ${\bm g}$ represents
the gravitational force. The bulk density is
$\rho = \phi \rho_f + (1-\phi)\rho_s$, where $\rho_s$ and $\rho_f$ are
the densities of solid and fluid phases and $\phi$ is the
porosity. Finally, the source term $f$ represents a forced fluid
extraction or injection process.

Our focus here is on the so-called three-field formulation in which Darcy's
velocity, ${\bm w}$, is also a primary unknown in addition to ${\bm u}$ and $p$.  
As a result, we have the following system of PDEs:
\begin{eqnarray} 
&& -\ddiv\bm{\sigma}' + \alpha \nabla p = \rho \bm{g},\qquad
\mbox{where} \quad
  \bm{\sigma}' = 2\mu  {\boldsymbol \varepsilon}(\bm{u})  + \lambda\ddiv(\bm{u}) \bm{I}, \label{three-field1}\\
&& {\bf K}^{-1}\mu_f {\bm w} + \nabla p = \rho_f {\bm g}, \label{three-field2}\\
&& \displaystyle \frac{\partial}{\partial t}\left( \frac{1}{M}p + \alpha \ {\ddiv} \, {\bm u}\right)  + {\ddiv} \, {\bm w} = f.
\label{three-field3}
\end{eqnarray}
This system is often subject to the following set of boundary conditions:
\begin{eqnarray}
&&  p = 0, \quad \mbox{for}\quad x\in\overline{\Gamma}_t,
\quad \boldsymbol \sigma' \, {\bm n} = {\bm 0}, \quad 
\quad \mbox{for}\quad x\in\Gamma _t, \label{bound-cond-t}\\
&& {\bm u} = {\bm 0}, \quad  \mbox{for}\quad x\in \overline{\Gamma}_c,\quad
{\bm w}\cdot {\bm n} = 0, \quad \mbox{for}\quad x\in \Gamma_c,
\label{bound-cond-c}
\end{eqnarray}
where ${\bm n}$ is the outward unit normal to the boundary,
$\overline{\Gamma}=\overline \Gamma_t \cup \overline \Gamma_c$, with
$\Gamma_t$ and $\Gamma_c$ being open (with respect to $\Gamma$) subsets of
$\Gamma$ with nonzero measure. In the following, we omit the symbol
``$\overline{~~}$'' over $\Gamma_t$ and $\Gamma_c$ as it will be clear
from the context that the essential boundary conditions are imposed on
closed subsets of $\Gamma$. Other sets of boundary conditions are also
of interest, such as full Dirichlet conditions for both $p$ and
$\bm{u}$ on all of $\Gamma$.

The initial condition at $t=0$ is given by,
\begin{equation}\label{ini-cond}
      \left( \frac{1}{M}p + \alpha \ddiv {\bm u} \right)\, (\bm{x},0)=0, \,  \bm{x} \in\Omega,
\end{equation}
which yields the following mixed formulation 
of Biot's three-field consolidation model:\\ 
For each $t\in (0,T]$, find
$(\bm u(t),\bm w(t), p(t))\in \bm V \times \bm W \times Q$ such that
\begin{eqnarray}
  && a(\bm{u},\bm{v}) - (\alpha p,\ddiv \bm{v}) = (\rho\bm g,\bm{v}),
     \quad \forall \  \bm{v}\in \bm V, \label{variational1}\\
  && ({\bm K}^{-1}\mu_f\bm{w},\bm{r}) - (p,\ddiv \bm{r}) = (\rho_f {\bm g}, \bm{r}), \quad \forall \ \bm{r}\in \bm W,\label{variational2}\\
  &&  \left(\frac{1}{M}\frac{\partial p}{\partial t} ,q \right)+\left(\alpha\ddiv \frac{\partial \bm{u}}{\partial t},q\right) + (\ddiv \bm{w},q)   = (f,q), \quad \forall \ q \in Q,\label{variational3}
\end{eqnarray}
where,
\begin{equation}\label{bilinear}
a(\bm{u},\bm{v}) = 
2\mu\int_{\Omega}{\boldsymbol \varepsilon}(\bm{u}):{\boldsymbol \varepsilon}(\bm{v}) +
\lambda\int_{\Omega} \ddiv\bm{u}\ddiv\bm{v},
\end{equation}
corresponds to linear elasticity. The function spaces used in the
variational form are
\begin{eqnarray*}
&&{\bm V} = \{{\bm u}\in {\bm H}^1(\Omega) \ |  \ {\bm u}|_{\overline{\Gamma}_c} = {\bm 0} \},\\
&&{\bm W} = \{{\bm w} \in \bm{H}(\ddiv,\Omega) \ | \  ({\bm w}\cdot {\bm n})|_{\Gamma_c} = 0\}, \\
&&Q = L^2(\Omega),
\end{eqnarray*} 
where ${\bm H}^1(\Omega)$ is the space of square integrable
vector-valued functions whose first derivatives are also square
integrable, and $\bm{H}(\ddiv,\Omega)$ contains the square integrable
vector-valued functions with square integrable divergence.

We recall that the well-posedness of the continuous problem was
established by Showalter~\cite{showalter}, and, for the three-field
formulation by Lipnikov~\cite{lipnikov_phd}. Next, we focus on the
behavior of some classical discretizations of Biot's model
.

\subsection{Discretizations}\label{section:discrete}
First, we partition the domain $\Omega$ into
$n$-dimensional simplices and denote the resulting partition with
$\mathcal T_h$, i.e.,
$\overline \Omega = \cup_{T\in \mathcal{T}_h} \overline T$. Further,
with every simplex $T\in \mathcal T_h$, we associate two quantities
which characterize its shape: the diameter of $T$,
$h_T=\operatorname{diam}(T)$, and the radius, $\rho_T$, of the
$n$-dimensional ball inscribed in $T$. The simplicial mesh is {\it
  shape regular} if and only if $h_T/\rho_T\lesssim 1$ uniformly with  
respect to $T$. 

With the partitioning, $\mathcal{T}_h$, we associate a triple of piecewise
polynomial, finite-dimensional spaces,
\begin{equation}\label{include}
\bm{V}_h\subset \bm{V}, \quad \bm{W}_h\subset \bm{W}, \quad 
Q_h \subset Q.
\end{equation}
While we specify two choices of the space $\bm{V}_h$ later, 
we fix $\bm{W}_h$ and $Q_h$ as follows,
\begin{eqnarray*}
  &&{\bm W}_h = \{{\bm w}_h \in \bm{W} \ | \  {\bm w}_h|_T = {\bm a} + \eta {\mathbf x}, \ {\bm a}\in {\mathbb R}^d, ~\eta\in {\mathbb R}, \ \ \forall\, T\in \mathcal{ T}_h \}, \\
  &&Q_h =\{q_h \in Q \ | \ q_h|_T \in {\mathbb P}_0(T), \ \ \forall\, T \in \mathcal{ T}_h \},
\end{eqnarray*} 
where ${\mathbb P}_0(T)$ is the one-dimensional space of constant
functions on $T$.  We note that the inclusions listed
in~\eqref{include} imply that the elements of $\bm{V}_h$ are
continuous on $\Omega$, the functions in $\bm{W}_h$ have continuous
normal components across element boundaries, and that the functions in $Q_h$
are in $L^2(\Omega)$.  This choice of $\bm{W}_h$ is the 
standard lowest order Raviart-Thomas-\Nedelec space (RT0) and $Q_h$ is 
the piecewise constant space (P0). 

Finally, using backward Euler as a time discretization on a
time interval $(0,t_{\max{}}]$ with constant time-step size
$\tau$, the discrete scheme corresponding to the three-field
formulation~\eqref{variational1}-\eqref{variational3} reads:\\
Find $(\bm u_h^m,\bm w_h^m, p_h^m)\in \bm V_h \times \bm W_h \times Q_h$
such that
\begin{eqnarray}
  && a(\bm{u}_h^m,\bm{v}_h) - (\alpha p_h^m,\ddiv \bm{v}_h) = (\rho\bm g,\bm{v}_h),
     \quad \forall \  \bm{v}_h\in \bm V_h, \label{discrete1}\\
  && \tau({\bm K}^{-1}\mu_f\bm{w}_h^m,\bm{r}_h) - \tau(p_h^m,\ddiv \bm{r}_h) = \tau(\rho_f {\bm g}, \bm{r}_h), \quad \forall \ \bm{r}_h\in \bm W_h,\label{discrete2}\\
  &&  \left(\frac{1}{M} p^m_h ,q_h \right)+\left(\alpha\ddiv \bm{u}_h^m,q_h\right) + \tau(\ddiv \bm{w}_h^m,q_h)   = (\widetilde{f}, q_h), \quad \forall \ q_h \in Q_h,\label{discrete3}
\end{eqnarray}
where
$(\widetilde{f}, q_h) =\tau(f,q_h) + \left(\frac{1}{M} p^{m-1}_h ,q_h
\right)+\left(\alpha\ddiv \bm{u}_h^{m-1},q_h\right)$,
and, 
\[
(\bm{u}_h^m, \bm{w}_h^m, p_h^m) \approx 
\left(\bm{u}(\cdot, t_m),\bm{w}(\cdot, t_m), p(\cdot, t_m)\right), 
\quad t_m = m\tau, \ m = 1,2,\ldots
\] 

 \subsection{Effects of permeability on the error of
   approximation}\label{sec:difficulties}
For $\bm{V}_h$, we start with a popular
finite-element approximation for~\eqref{variational1}--\eqref{variational3} by choosing
\begin{eqnarray*}
\bm{V}_h=\bm{V}_{h,1},\quad \mbox{with}\quad
\bm{V}_{h,1} &:=& \{{\bm v}_h\in {\bm V} \; \big|  \; {\bm v}_h|_T \in [{\mathbb P}_1(T)]^d,~\mbox{for all}~T\in \mathcal{ T}_h \},
\end{eqnarray*} 
where $\mathbb{P}_1(T)$ is the space of linear polynomials on $T\in \mathcal{T}_h$. 
Then, $\bm{V}_{h,1}$  is the space of
 piecewise linear (with respect to $\mathcal{T}_h$), continuous
 vector-valued functions.
For uniformly positive definite permeability tensor, $\bm{K}$,
such choice of spaces has been successfully employed for
numerical simulations of Biot's consolidation model
(see~\cite{phillips2, lipnikov_phd}). However, the heuristic considerations that expose some of the issues with this
discretization are observed in cases when 
$\bm{K}\to \bm{0}$. In such cases, $\bm w\to\bm 0$ and the
discrete problem approaches a P1-P0 discretization of the
Stokes' equation. As it is well known, the element pair,
$\bm{V}_{h,1}\times Q_h$, does not satisfy the inf-sup
condition and is unstable for the Stokes' problem.  In fact, on a
uniform grid in $2D$, it is easy to prove that
volumetric locking occurs, namely,
that the only divergence-free function from $\bm{V}_{h,1}$ is the zero
function.  More precisely,
\[
\operatorname{dim}(Q_h) > 
\operatorname{dim} \bm{V}_h >
\operatorname{dim} \operatorname{Range}(\operatorname{div}_h), 
\quad \operatorname{div}_h = \operatorname{div}\big|_{\bm{V}_h}.
\]
These inequalities imply that $\operatorname{div}_h$ is not an onto
operator, and, hence, the pair of spaces violates the inf-sup
condition~~\cite{1991BrezziF_FortinM-aa} associated with the discrete
Stokes' problem.  More details on this undesirable phenomenon for Stokes are found in the classical
monograph~\cite{1991BrezziF_FortinM-aa} and also in
\cite[pp.~45--100]{2008BoffiD_BrezziF_DemkowiczL_DuranR_FalkR_FortinM-aa}
and \cite{2013BoffiD_BrezziF_FortinM-aa}.

Here, we demonstrate numerically that for Biot's model, the error in
the finite-element approximation does not decrease when the permeability is small
relative to the mesh size. We consider 
$\Omega=(0,1)\times (0,1)$, and approximate
\eqref{variational1}-\eqref{variational3} subject to Dirichlet boundary
conditions for both $\bm{u}$ and $p$ on the whole of $\Gamma$. 
We cover $\Omega$ with a uniform triangular grid by 
dividing an $(N\times N)$ uniform 
square mesh into right triangles.  The material parameters are 
$\lambda = 2$, $\mu = 1$, $\mu_f = 1$, $\alpha = 1$, and $M = 10^6$. We consider
a diagonal permeability tensor ${\bm K} = k {\bm I}$ with constant
$k$, and the other data is set so that the exact solution is given by
\begin{eqnarray*}
\bm{u}(x,y,t) &=& \operatorname{curl} \varphi = 
\begin{pmatrix}\partial_y \varphi \\ -\partial_x \varphi \end{pmatrix}, 
\quad \varphi (x,y) = [xy(1-x)(1-y)]^2, \\
p(x,y,t) &=& 1.
\end{eqnarray*}
Finally, we set $\tau = 1$ and $t_{\max{}}=1$, so that we only perform one time step. 

As seen in
Table~\ref{errors_difficulties} the energy norm ($\| \bm{v} \|_A^2 := a(\bm{v}, \bm{v}) $ for $\bm{v} \in \bm{V}$) for the displacement
errors and the $L^2$-norm for
pressure errors do not decrease until the mesh size
is sufficiently small (compared with the permeability). Thus for small
permeabilities, this could
result in expensive discretizations which are less applicable to
practical situations.
\begin{table}[htb!]
\begin{center}
\begin{tabular}{|c|c|c|c|c|c|c|}
\cline{1-6}
\diagbox{$\kappa$}{$N$}&
$8$ & $16$ & $32$ & $64$ & $128$ &\multicolumn{1}{c}{~} \\
\hline
\multirow{2}{*}{$10^{-4}$} 
& $0.0187$ & $0.0040$ & $0.0009$ & $0.0002$ & $5.66\times 10^{-5}$ 
& $\|\Pi_1{\bm u}-{\bm u}_h\|_A$ \\
& $0.0590$ & $0.0090$ & $0.0016$ & $0.0003$ & $8.33\times 10^{-5}$ & $\|\Pi_0p-p_h\|_{L^2}$ \\
\hline
\multirow{2}{*}{$10^{-6}$} 
& $0.0547$ & $0.0302$ & $0.0050$ & $0.0005$ & $8.93\times 10^{-5}$ & $\|\Pi_1{\bm u}-{\bm u}_h\|_A$ \\
& $0.3187$ & $0.3098$ & $0.0741$ & $0.0097$ & $0.0012$ & $\|\Pi_0p-p_h\|_{L^2}$ \\
\hline
\multirow{2}{*}{$10^{-8}$} 
& $0.0578$ & $0.0567$ & $0.0476$ & $0.0165$ & $0.0018$ & $\|\Pi_1{\bm u}-{\bm u}_h\|_A$ \\ 
& $0.3388$ & $0.7067$ & $1.1418$ & $0.6450$ & $0.1142$ & $\|\Pi_0p-p_h\|_{L^2}$  \\
\hline
\multirow{2}{*}{$10^{-10}$} 
& $0.0578$ & $0.0574$ & $0.0570$ & $0.0550$ & $0.0502$ & $\|\Pi_1{\bm u}-{\bm u}_h\|_A$ \\ 
& $0.3372$ & $0.7176$ & $1.4527$ & $2.7790$ & $3.4403$ & $\|\Pi_0p-p_h\|_{L^2}$ \\
\hline
\end{tabular}
\caption{Energy norm and $L^2$-norm for displacement and pressure
  errors, respectively. Hydraulic conductivity: $\kappa
  =(k/\mu_f)=10^{-2\nu}$, $\nu = 2,\ldots,5$; grid parameters:
  $h=2^{-l}$ and $N=2^l$, $l=3,\ldots,7$.  Results confirm poor
 approximation when $\kappa/h$ is small, where $\kappa = k/\mu_f$.
\label{errors_difficulties}}
\end{center}
\end{table}

\section{Stabilization and perturbation of the bilinear form}\label{sec:stab}
To resolve the above issue, we introduce a well-known stabilization technique
based on enrichment of the piecewise linear continuous finite-element
space, $\bm{V}_{h,1}$, with edge/face (2D/3D) bubble functions
(see~\cite[pp. 145-149]{GR1986}).  The discretization described below is
based on a Stokes-stable pair of spaces $(\bm{V}_h,Q_h)$ with
$\bm{V}_h \supset \bm{V}_{h,1}$.  As we show later, in
Section~\ref{sec:no_bubbles}, this stabilization gives a proper finite-element approximation of the solution of Biot's model
independently of the size of the hydraulic conductivity, $\kappa$.

\subsection{Stabilization by face bubbles}
To define the enriched space, following~\cite{GR1986}, consider
the set of $(d-1)$ dimensional faces from $\mathcal{T}_h$ and denote
this set by $\mathcal{E}= \mathcal{E}^{o}\cup\mathcal{E}^{\partial}$,
where $\mathcal{E}^{o}$ is the set of interior faces (shared by two
elements) and $ \mathcal{E}^{\partial}$ is the set of faces on the
boundary.  In addition, $\mathcal{E}^{\Gamma_t}$ is the set of faces on the boundary $\Gamma_t$ and $\mathcal{E}^{o,t} = \mathcal{E}^{o} \cap \mathcal{E}^{\Gamma_t}$.  
Note, if $\Gamma_t = \partial \Omega$ (pure traction boundary condition), then $\mathcal{E}^{\Gamma_t} = \mathcal{E}^{\partial}$ and $\mathcal{E}^{o,t} = \mathcal{E}$.
For any face $e\in \mathcal{E}^{o}$, such that
$e\in \partial T$, and $T\in \mathcal{T}_h$, let $\bm{n}_{e,T}$ be the
outward (with respect to $T$) unit normal vector to $e$.  With every
face $e\in \mathcal{E}^{o}$, we also associate a unit vector $\bm{n}_e$
which is orthogonal to it. Clearly, if $e\in \partial T$ we have
$\bm{n}_e = \pm \bm{n}_{e,T}$.  For the boundary faces
$e\in \mathcal{E}^{\partial}$, we always set $\bm{n}_e=\bm{n}_{e,T}$,
where $T$ is the \emph{unique} element for which we have
$e\subset \partial T$.  For the interior faces, the particular
direction of $\bm{n}_e$ is not important, although it is important
that this direction is fixed.  More precisely,
\begin{equation}
  \bm{n}_e = \bm{n}_{e,T^{+}} = -\bm{n}_{e,T^{-}}
  \quad\mbox{if}\quad e=T^{+}\cap T^{-}, \quad \mbox{and}\quad T^{\pm}\in \mathcal{T}_h,
\end{equation}
Further, with every face
$e\in \mathcal{E}$, $e=T^{+}\cap T^{-}$, we associate a 
vector-valued function $\bm{\Phi}_{e}$,
\begin{equation}\label{e:def-phie}
\bm{\Phi}_{e} = \varphi_e\bm{n}_{e}, 
\quad\mbox{with}\quad 
\varphi_e\bigg|_{T^{\pm}} = \varphi_{e,T^{\pm}}, 
\quad\mbox{and}\quad 
\varphi_{e,T^{\pm}} = \prod_{k=1, k\neq j^{\pm}}^{d+1}\lambda_{k,T^{\pm}},
\end{equation}
where $\lambda_{k,T^{\pm}}$, $k=1,\ldots,(d+1)$ are barycentric
coordinates on $T^{\pm}$ and $j^{\pm}$ is the vertex opposite to the face
$e$ in $T^{\pm}$.  We note that $\bm{\Phi}_e\in {\bm V}$ is a
continuous piecewise polynomial function of degree $d$.

Finally, the stabilized finite-element space $\bm{V}_h$ is defined as
\begin{equation}\label{space_u}
{\bm V}_h = \bm{V}_{h,1} \oplus\bm{V}_b, \quad 
\bm{V}_b=\operatorname{span}\{\bm{\Phi}_e\}_{e\in  \mathcal{E}^{o,t}}. 
\end{equation}
The degrees of freedom associated with $\bm{V}_h$ are the values
at the vertices of $\mathcal{T}_h$ and the total flux through
$e\in \mathcal{E}^{o,t}$ of $(I-\Pi_1){\bm{v}}_h$, where $\Pi_1$ is the standard piecewise linear interpolant,
$\Pi_1: C(\overline{\Omega})\mapsto \bm{V}_{h,1}$. Then, the canonical interpolant,
$\Pi: C(\overline{\Omega})\mapsto \bm{V}_h$, is defined as:
\[
\Pi \bm{v} = \Pi_1 \bm{v} + \sum_{e\in \mathcal{E}^{o,t}} v_e \bm{\Phi}_e,\quad
v_e = \frac{1}{|e|}\int_e (I-\Pi_1)\bm{v}.
\]
With this choice of $\bm{V}_h$,  the variational
form,~\eqref{discrete1}--\eqref{discrete3}, remains the same and we have the following block form of the
discrete problem:
\begin{equation}\label{block_form}
\mathcal{ A} \left(
\begin{array}{c} 
{\bm U}_b \\ 
{\bm U}_l \\ 
{\bm W} \\ 
{\bm P} 
\end{array}
\right) = 
{\bm b}, \ \ \hbox{with} \ \ 
\mathcal{ A} = \left( 
\begin{array}{cccc} 
A_{bb} & A_{bl} & 0 & G_b \\ 
A_{bl}^T & A_{ll} & 0 & G_l \\ 
0 & 0 & \tau M_w & \tau G \\
G_b^T & G_l^T & \tau G^T & -M_p
\end{array} 
\right),
\end{equation}
where ${\bm U}_b$, ${\bm U}_l$, ${\bm W}$ and ${\bm P}$ are the
unknown vectors for the bubble components of the displacement, the
piecewise
linear components of the displacement, the Darcy velocity, and the
pressure, respectively. The 
blocks in the definition of $\mathcal{A}$ correspond to the following
bilinear forms: 
\begin{eqnarray*}
  && a(\bm{u}_h^b,\bm{v}_h^b) \rightarrow A_{bb}, \quad a(\bm{u}_h^l,\bm{v}_h^b) \rightarrow A_{bl}, \quad  a(\bm{u}_h^l,\bm{v}_h^l) \rightarrow A_{ll}, \\
   &&-(\alpha p_h,\ddiv \bm{v}_h^b) \rightarrow G_b,  \quad   -(\alpha
   p_h,\ddiv \bm{v}_h^l) \rightarrow G_l, \quad -(p_h,\ddiv \bm{r}_h) \rightarrow G,\\
  && (K^{-1}\mu_f\bm{w}_h,\bm{r}_h) \rightarrow  M_w, \quad
 \left(\frac{1}{M} p_h ,q_h \right) \rightarrow M_p, 
\end{eqnarray*}
where ${\bm u}_h = {\bm u}_h^l+{\bm u}_h^b$,
${\bm u}_h^l\in \bm{V}_{h,1}$, ${\bm u}_h^b\in \bm{V}_b$, and an analogous
decomposition for $\bm{v}_h$.


Next, we define the following notion of stability for discretizations of Biot's
model needed for the analysis.
\begin{definition}\label{d:stokes-biot-stable}
  The triple of spaces $(\widetilde{\bm{V}}_h,\widetilde{\bm{W}}_h,\widetilde{Q}_h)$ is
  \emph{Stokes-Biot stable} if and only if the following conditions
  are satisfied:
\begin{itemize}
\item $a(\bm{u}_h, \bm{v}_h)  \leq C_{\bm{V}} \| \bm{u}_h \|_1 \| \bm{v}_h \|_1$, for all $\bm{u}_h\in \widetilde{\bm{V}}_h$,
  $\bm{v}_h \in \widetilde{\bm{V}}_h$;
\item $a(\bm{u}_h, \bm{u}_h)  \geq \alpha_{\bm{V}} \| \bm{u}_h \|_1^2$, for all $\bm{u}_h\in \widetilde{\bm{V}}_h$;
\item The pair of spaces $(\widetilde{\bm{W}}_h,\widetilde{Q}_h)$ is
  Poisson stable, i.e., it satisfies stability and continuity
  conditions required by the mixed discretization of the Poisson
  equation;
\item The pair of spaces $(\widetilde{\bm{V}}_h,\widetilde{Q}_h)$ is Stokes stable. 
\end{itemize}
Here, $\| \cdot \|_1$ and $\| \cdot \|$ denote the standard
${\bm H}^1$ norm and $L^2$ norm, respectively.
\end{definition}
To show how stability for the Biot's system follows from the conditions above, 
we introduce a norm on $\bm{V}_h\times {\bm W}_h \times Q_h$:
\begin{equation}\label{norm}
\|({\bm u}_h,{\bm w}_h,p_h) \|_{\tau} = \left( \| {\bm u}_h \|_1^2 + \tau \| {\bm w}_h \|^2 + \tau^2 \| \ddiv {\bm w}_h\|^2 + \| p_h \|^2 \right)^{1/2},
\end{equation}
Further, we
associate a composite bilinear form on the space,
$\bm{V}_h\times {\bm W}_h \times Q_h$,
\begin{eqnarray*}
B(\bm{u}_h, \bm{w}_h, p_h; \bm{v}_h, \bm{r}_h, q_h) & := & 
a(\bm{u}_h,\bm{v}_h) - (\alpha p_h,\ddiv \bm{v}_h)  + \tau(K^{-1}\mu_f\bm{w}_h,\bm{r}_h)  - \tau(p_h,\ddiv \bm{r}_h) \\
&& \quad - \left(\frac{1}{M} p_h ,q_h \right)-\left(\alpha\ddiv \bm{u}_h,q_h\right) - \tau(\ddiv \bm{w}_h,q_h).
\end{eqnarray*}
We then have the following theorem which shows that on every time step
the discrete problem is solvable. 
\begin{theorem}\label{t:stable1} 
If the triple $(\bm{V}_h, \bm{W}_h, Q_h)$ is Stokes-Biot stable, then: 
\begin{itemize}
\item[] $B(\cdot, \cdot, \cdot \ ; \ \cdot, \cdot, \cdot)$ is
  continuous with respect to $\|(\cdot,\cdot,\cdot)\|_\tau$; and
\item[] the following inf-sup condition holds.
\begin{equation}\label{infsup-B}
\sup_{(\bm{v}_h, \bm{r}_h, q_h) \in {\bm V}_h\times \bm W_h\times Q_h}  \frac{B(\bm{u}_h, \bm{w}_h, p_h; \bm{v}_h, \bm{r}_h, q_h)}{\|  \left( \bm{v}_h,  \bm{r}_h,  q_h  \right) \|_{\tau}} \geq \gamma \| \left( \bm{u}_h,  \bm{w}_h, p_h \right) \|_{\tau},
\end{equation}
with a constant $\gamma > 0$ independent of mesh size $h$ and time
step size $\tau$.
\end{itemize}
\end{theorem}
\begin{proof}
  Using the conditions in Definition~\ref{d:stokes-biot-stable}, the
  two items in the statement of this theorem follow from arguments
  identical to the ones given in the proof of Theorem~2
  in~\cite{2017HuX_RodrigoC_GasparF_ZikatanovL-aa}.
\end{proof}
Note that if we replace $a(\cdot,\cdot)$ with any spectrally
equivalent bilinear form on $\bm{V}_h\times \bm{V}_h$, the same
stability result holds true. In the next section, we introduce such
a spectrally equivalent bilinear form which allows for: (1) Efficient
elimination of the degrees of freedom corresponding to the bubble
functions via static condensation; and (2) Derivation of optimal error
estimates for the fully discrete problem, following the analysis
in~\cite{2017HuX_RodrigoC_GasparF_ZikatanovL-aa}.  

\section{Local perturbation of the bilinear form and elimination of bubbles}
\label{sec:no_bubbles}

 
In this section, we show how the unknowns corresponding to
degrees of freedom in $\bm{V}_b$ can be eliminated. A straightforward
elimination of the edge/face bubbles is not local, and, in general,
leads to prohibitively large number of non-zeroes in the resulting
linear system.  To resolve this, we introduce a
consistent perturbation of $a(\cdot,\cdot)$, which has a diagonal matrix
representation. It is then easy to eliminate the unknowns
corresponding to the bubble functions in $\bm V_b$ with no fill-in.
This leads to a stable P1-RT0-P0 discretization for the Biot's model
and, consequently, to a stable P1-P0 discretization for the Stokes' equation. 

First, consider a natural decomposition of
$\bm{u}\in {\bm V}_h$:
\begin{equation}\label{natural-decomposition}
\bm{u} = 
\bm{u}_{l} +  \bm{u}_{b} = \underbrace{\Pi_1 \bm{u}}_{\bm{u}_l} + \underbrace{\sum_{e\in \mathcal E^{o,t}} u_e\bm{\Phi}_e}_{\bm{u}_b},
\end{equation}
and the local bilinear forms for $T\in \mathcal{T}_h$,
$\bm{u}\in \bm{V}_h$, and $\bm{v}\in \bm{V}_h$:
\begin{equation}\label{bilinearT}
a_T(\bm{u},\bm{v}) = 
2\mu\int_{T}\varepsilon(\bm{u}):\varepsilon(\bm{v}) +
\lambda\int_{T} \ddiv\bm{u}\ddiv\bm{v}.
\end{equation}
For the restriction of $a(\cdot,\cdot)$ onto the space spanned by bubble
functions $\bm{V}_b$, we have
\[
a_b(\bm u_b, \bm v_b) := a(\bm u_b, \bm v_b)  = \sum_{T\in
  \mathcal{T}_h} a_{b,T}(\bm u_b, \bm v_b) = \sum_{T\in \mathcal{T}_h}
\sum_{e,e' \in \partial T}u_e v_{e'} a_T(\bm{\Phi}_{e'}, \bm{\Phi}_{e}). 
\]
On each element, $T\in\mathcal{T}_h$, then introduce 
\begin{equation}\label{e:db}
d_{b,T}(\bm{u},\bm{v}) = (d+1)\sum_{e\in \partial T} u_ev_ea_T(\bm{\Phi}_e,\bm{\Phi}_e), \quad 
d_{b}(\bm{u},\bm{v}) = \sum_{T\in \mathcal{T}_h} d_{b,T}(\bm{u},\bm{v}).
\end{equation}
Replacing $a_b(\cdot,\cdot)$ with $d_b(\cdot,\cdot)$ gives a
perturbation, $a^{D}(\cdot,\cdot)$, of $a(\cdot,\cdot)$:
\begin{equation}\label{ad-form}
a^{D}(\bm u, \bm v) := d_{b}(\bm u_b, \bm v_b) + 
a(\bm u_b, \bm v_l) + a(\bm u_l, \bm v_b) + a(\bm u_l, \bm v_l)
\end{equation}

\subsection{A spectral equivalence result}
To prove that the form $a^D(\cdot,\cdot)$ and
$a(\cdot,\cdot)$ are spectrally equivalent, we need several auxiliary
results.  First, recall the
definition of the rigid body motions (modes), $\mathfrak{R}$ on
$\mathbb{R}^d$:
\[
\mathfrak{R} = \left\{ {\bm v} = {\bm a} + \mathfrak{b}{\bm x}\;\big|\;
{\bm a}\in\mathbb{R}^d,\quad 
\mathfrak{b}\in \mathfrak{so}(d)\right\},
\] 
where $\mathfrak{so}(d)$ is the algebra of skew-symmetric
$(d\times d)$ matrices.  The dimension of $\mathfrak{R}$
is $\frac12d(d+1)$ and its elements are component-wise linear
vector-valued functions. 

Next, recall the classical Korn
inequality~\cite{1908KornA-aa,1909KornA-aa} for
$\bm u\in {\bm H}^1(Y)$ for a domain $Y \subset \mathbb{R}^d$,
star-shaped with respect to a ball. As shown by
Kondratiev and Oleinik in~\cite{1989KondratievV_OleinikO-aa,1989KondratievV_OleinikO-ab},
\begin{equation}\label{Korn1}
\inf_{\mathfrak{m}\in \mathfrak{so}(d)} \|\nabla {\bm u} - \mathfrak{m}\|_{L^2(Y)} \lesssim
\|\varepsilon({\bm u})\|_{L^2(Y)}, 
\end{equation}
where the constant hidden in $\lesssim$ depends on the shape
regularity of $Y$, that is, on the ratio
$\displaystyle \frac{\operatorname{diam}(Y)}{R}$.
For convenience when referencing \eqref{Korn1} later, we state the
following lemma, which gives a simpler version of the inequality
defined on simplices, where $Y=T\in\mathcal{T}_h$. 
\begin{lemma} \label{lemma_korn}
Let $\mathcal{T}_h$ be a shape-regular simplicial mesh covering $\Omega$. 
Then, the following inequality holds for any 
$T\in \mathcal T_h$ and  $\bm{u}\in {\bm H}^1(T)$:
\begin{equation}\label{e:korn-simplex}
\inf_{\mathfrak{m}\in\mathfrak{so}(n)}\|\nabla{\bm u}-\mathfrak{m}\|_{L^2(T)} 
\lesssim \|\varepsilon({\bm u})\|_{L^2(T)}, 
\end{equation}
where the constant hidden in ``$\lesssim$'' 
depends on the shape regularity constant of $\mathcal{T}_h$. 
\end{lemma}

Defining the unscaled bilinear form, $\widetilde{d}_{b,T}$,
\begin{equation}
\widetilde{d}_{b,T}(\bm{u},\bm{v}) := \sum_{e\in \partial T} u_ev_ea_{T}(\bm{\Phi}_e,\bm{\Phi}_e), 
\end{equation}
we have the following local, spectral equivalence result. 
\begin{lemma}\label{l:local-spectral-equivalence}
  For all $T\in \mathcal{T}_h$  the following inequalities hold:
\begin{equation}\label{lemma_RBM}
\eta_T\widetilde{d}_{b,T}({\bm u},{\bm u}) \le 
a_{b,T}({\bm u},{\bm u})\le 
(d+1)\widetilde{d}_{b,T}({\bm u},{\bm u}), 
\quad \mbox{for all}\quad {\bm u}\in \bm{V}_b,
\end{equation}
where the constant $\eta_T$ is independent of $h_T$ and $\rho_T$.
\end{lemma}
\begin{proof}
  Set $a_{ee'} = a_{b,T}(\bm{\Phi}_{e'},\bm{\Phi}_e)$ and note that
  $a_{ee} = \widetilde{d}_{b,T}(\bm{\Phi}_e,\bm{\Phi}_e)$ for all
  $e,\;e'\in \partial{T}$.  
  The upper bound follows immediately by several (two) applications of
  the Cauchy-Schwarz inequality:
\begin{eqnarray*}
a_{b,T}({\bm u}, {\bm u}) &=& \sum_{e,e'\in \partial T} a_{ee'}u_e u_{e'}
\leq  \sum_{e,e'\in \partial T} \sqrt{a_{ee}a_{e'e'}}|u_eu_{e'}| 
 =\left( \sum_{e\in \partial T} \sqrt{a_{ee}} |u_e|   \right)^2\\
& \leq & (d+1) \sum_{e\in \partial T}a_{ee} u_e^2 = (d+1)\widetilde{d}_{b,T}({\bm u},{\bm u}). 
\end{eqnarray*}
We prove the lower bound by establishing the following inequalities
for $\bm{u}\in \bm{V}_b$:
\begin{equation}\label{e:two-inequalities}
h_T^{-2} \|{\bm u}\|^2_{L^2(T)} \lesssim a_{b,T}({\bm u},{\bm u}), 
\quad\mbox{and}\quad
\widetilde{d}_{b,T}({\bm u},{\bm u}) \lesssim h_T^{-2} \|{\bm u}\|^2_{L^2(T)}.
\end{equation}
By definition 
for all $\bm{u}\in \bm{V}_b$ and all rigid body modes ${\mathfrak r}\in \mathfrak{R}$, we have
that $\Pi_1{\bm u}=\bm{0}$ and $\Pi_1 {\mathfrak r} = {\mathfrak r}$. The classical
interpolation estimates found in~\cite[Chapter 3]{1978CiarletP-aa} give
\begin{eqnarray*}
\| {\bm u}\|^2_{L^2(T)} & = & 
\| {\bm u} - {\mathfrak r} - \Pi_1({\bm u}-{\mathfrak r}) \|^2_{L^2(T)} 
\lesssim h_T^2 \|\nabla({\bm u} - {\mathfrak r}) \|^2_{L^2(T)}.
\end{eqnarray*}
Taking the infimum over all $\mathfrak{r}\in \mathfrak{R}$ and applying Korn's inequality 
(Lemma~\ref{lemma_korn}) then yields
\begin{eqnarray*}
h_T^{-2} \| {\bm u}\|^2_{L^2(T)} & \lesssim&  
 \inf_{{\mathfrak r}\in \mathfrak{R}}\|\nabla({\bm u} - {\mathfrak r})\|^2_{L^2(T)}
 =  \inf_{\mathfrak{m}\in \mathfrak{so}(d)}\|\nabla{\bm u} - \mathfrak{m}\|^2_{L^2(T)} \lesssim 
\|\varepsilon({\bm u})\|^2_{L^2(T)}.
\end{eqnarray*}
This shows the first inequality in~\eqref{e:two-inequalities}, and  
to prove the second inequality, we note that from the definition of 
$\widetilde{d}_{b,T}(\cdot,\cdot)$ and the inverse inequality, we have that 
\[
\widetilde{d}_{b,T}(\bm u,\bm u)
\lesssim \sum_{e\in \partial T}  u_e^2 \left[\|\nabla \bm{\Phi}_e\|_{L^2(T)}^2 + \lambda \|\ddiv \bm{\Phi}_e\|_{L^2(T)}^2\right] \lesssim h_T^{-2} \sum_{e\in \partial T}  u_e^2 \|\bm{\Phi}_e\|_{L^2(T)}^2.
\]
Recalling the 
definition of $\Phi_e$ in~\eqref{e:def-phie}
and the formula for integrating powers of the barycentric coordinates, gives
\begin{equation}\label{int-on-T}
\bm{\Phi}_e = \varphi_e \bm{n}_e, \quad 
\int_T \lambda_1^{\beta_1}\ldots\lambda_{d+1}^{\beta_{d+1}}\;dx =
|T|\frac{\beta_{1}!\ldots\beta_{d+1}! d! }{(\beta_{1}+\ldots+\beta_{d+1}+d)!}.
\end{equation}
It follows that \(\|\bm{\Phi}_e\|^2_{L^2(T)} = c_d|T|\)
and
\(\int_T \bm{\Phi}_e\bm{\Phi}_{e'}=  \frac{1}{2} c_d|T|(\delta_{ee'}+\bm{n}_e\cdot
\bm{n}_{e'})\),
with $c_d=\frac{d! \ 2^d}{(3d)!}$. As the Gram matrix
$(\bm{n}_e\cdot\bm{n}_{e'})_{e,e'\in\partial T}$ is positive semi-definite,
\begin{eqnarray*}
\sum_{e\in \partial T}  u_e^2 \|\bm{\Phi}_e\|_{L^2(T)}^2  & = & 
c_d |T|\sum_{e\in \partial T}  u_e^2 
  \le 
c_d|T|\left[
\sum_{e\in \partial T}  u_e^2 + \sum_{e,e'\in \partial T}  u_eu_{e'} (\bm{n}_e\cdot\bm{n}_{e'})\right] \\
& = &   
\left\|\sum_{e\in \partial T}  u_e\bm{\Phi}_e\right\|_{L^2(T)}^2 = \|\bm u\|_{L^2(T)}^2.
\end{eqnarray*}
Multiplying by $h_T^{-2}$ on both sides of this inequality furnishes
the proof of~\eqref{e:two-inequalities}, completing the proof of the Lemma.  
\end{proof}
Next, we show the spectral equivalence for the bilinear forms
$a(\cdot,\cdot)$ and $a^{D}(\cdot,\cdot)$.
\begin{lemma}\label{l:global-spectral-equivalence}
The following inequalities hold:
\[
a(\bm u, \bm u) \le a^{D}(\bm u, \bm u) \le \eta a(\bm u, \bm u), \quad \text{for all} \quad \bm{u} \in \bm{V}_h,
\]
where $\eta$ depends on the shape regularity of the mesh. 
\end{lemma}
\begin{proof}
Let $\bm{u}\in \bm{V}_h$, $\bm u = \bm u_l + \bm u_b$. 
From the definition of $d_{b}(\cdot,\cdot)$ in~\eqref{e:db}, 
$a_{b,T}(\bm u_b,\bm u_b) \le d_{b,T}(\bm u_b,\bm u_b)$, and the lower bound follows immediately:
\[
a(\bm{u},\bm{u}) -a^D(\bm{u},\bm{u}) =  
a_b(\bm{u}_b,\bm{u}_b) -d_b(\bm{u}_b,\bm{u}_b) = 
\sum_{T\in \mathcal{T}_h}[a_{b,T}(\bm{u}_b,\bm{u}_b) - d_{b,T}(\bm{u}_b,\bm{u}_b)] \le 0. 
\]
To prove the upper bound, we use the following local estimate, which
is established using an inverse inequality, a standard interpolation estimate, and $\Pi_1 \mathfrak{r} = \mathfrak{r}$ for all rigid body modes $\mathfrak{r} \in \mathfrak{R}$,
\begin{eqnarray*}
a_{T}({\bm u}_b,{\bm u}_b) & \lesssim & \|\nabla\bm u_b\|_{L^2(T)}^2 \lesssim h_T^{-2}\|\bm u_b\|_{L^2(T)}^2 = 
h_T^{-2}\|\bm u-\Pi_1 \bm u\|_{L^2(T)}^2 \\
& = & h_T^{-2}\|\bm (u - \mathfrak{r})-\Pi_1 (\bm u - \mathfrak{r})\|_{L^2(T)}^2  \lesssim   \|\nabla (\bm u  -\mathfrak{r})\|_{L^2(T)}^2.
\end{eqnarray*}
Taking the infimum over all $\mathfrak{r} \in \mathfrak{R}$ and applying the Korn's inequality (Lemma~\ref{lemma_korn}) then yields
\begin{equation*}
a_{T}({\bm u}_b,{\bm u}_b) \lesssim \inf_{\mathfrak{r} \in \mathfrak{R}} \| \nabla (\bm u - \mathfrak{r}) \|^2_{L^2(T)}  = \inf_{\mathfrak{m} \in \mathfrak{so}(d)} \| \nabla \bm u - \mathfrak{m} \|^2_{L^2(T)} \lesssim  \|\varepsilon({\bm u})\|^2_{L^2(T)}.
\end{equation*}

This inequality, combined with the  definition of $a^{D}(\cdot,\cdot)$, and 
the lower bound in Lemma~\ref{l:local-spectral-equivalence} gives,
\begin{eqnarray*}
  a^D(\bm{u},\bm{u}) & =  & a(\bm u,\bm u) + 
  \sum_{T\in \mathcal{T}_h} d_{b,T} (\bm u_b,\bm u_b) - a_{T} (\bm u_b,\bm u_b) \\
& \le & a(\bm u,\bm u)+ \sum_{T\in \mathcal{T}_h} \left(\frac{d+1}{\eta_T}-1\right)a_{T} (\bm u_b,\bm u_b) \\
& \lesssim & a(\bm u,\bm u) +  \sum_{T\in \mathcal{T}_h} \left(\frac{d+1}{\eta_T}-1\right)  \|\varepsilon(\bm u)\|_{L^2(T)}^2    \\
& \lesssim &  a(\bm u, \bm u).
\end{eqnarray*}
\end{proof}

Since we have shown that the bilinear form $a^D(\cdot,\cdot)$ can
replace $a(\cdot,\cdot)$ in Definition~\ref{d:stokes-biot-stable}, then
Theorem~\ref{t:stable1} holds when the bilinear form,
$B(\cdot, \cdot, \cdot \ ; \ \cdot, \cdot, \cdot)$, has 
$a^D(\cdot,\cdot)$ instead of $a(\cdot,\cdot)$. Thus, the variational problem, 
\begin{eqnarray}
  && a^D(\bm{u}_h^m,\bm{v}_h) - (\alpha p_h^m,\ddiv \bm{v}_h) = (\rho\bm g,\bm{v}_h),
     \quad \forall \  \bm{v}_h\in {\bm V}_h, \label{discrete1s_D}\\
  && (K^{-1}\mu_f\bm{w}_h^m,\bm{r}_h) - (p_h^m,\ddiv \bm{r}_h) = (\rho_f {\bm g}, \bm{r}_h), \quad \forall \ \bm{r}_h\in \bm W_h,\label{discrete2s_D}\\
  &&  \left(\frac{1}{M} \bar{\partial}_t p^m_h ,q_h \right)+\left(\alpha\ddiv \bar{\partial}_t \bm{u}_h^m,q_h\right) + (\ddiv \bm{w}_h^m,q_h)   = (f, q_h), \quad \forall \ q_h \in Q_h,\label{discrete3s_D}
\end{eqnarray}
has a unique solution and defines an invertible operator with inverse bounded independent of the mesh size $h$.  
This observation plays a crucial role in the error estimates in the
next subsection. 

For later comparison, we define following block form of the
discrete problem:
\begin{equation}\label{diagonal_block_form}
\mathcal{ A}^D \left(
\begin{array}{c} 
{\bm U}_b \\ 
{\bm U}_l \\ 
{\bm W} \\ 
{\bm P} 
\end{array}
\right) = 
{\bm b}, \ \ \hbox{with} \ \ 
\mathcal{ A}^D = \left( 
\begin{array}{cccc} 
D_{bb} & A_{bl} & 0 & G_b \\ 
A_{bl}^T & A_{ll} & 0 & G_l \\ 
0 & 0 & \tau M_w & \tau G \\
G_b^T & G_l^T & \tau G^T & -M_p
\end{array} 
\right),
\end{equation} 
where everything is defined as before and $D_{bb}$ corresponds to $a^D(\bm{u}_h^b,\bm{v}_h^b)$.
\subsection{Error estimates for the fully discrete problem}\label{sec:error}
To derive the error analysis of the fully discrete scheme, following the standard error
analysis of time-dependent problems in
Thom\'{e}e~\cite{2006ThomeeV-aa}, we first define the following
elliptic projections $\bar{\bm{u}}_h \in {\bm{V}}_h$,
$\bar{\bm{w}}_h \in \bm{W}_h$, and $\bar{p}_h \in Q_h$ for $t>0$ as
usual,
\begin{align}
& a^D(\bar{\bm{u}}_h, \bm{v}_h) - (\alpha\bar{p}_h, \ddiv \bm{v}_h) =
a(\bm{u}, \bm{v}_h) - (\alpha p, \ddiv \bm{v}_h),
\quad \forall \bm{v}_h \in {\bm{V}}_h, \label{eqn:elliptic-proj-u}\\
& (K^{-1}\mu_f \bar{\bm{w}}_h, \bm{r}_h ) - (\bar{p}_h, \ddiv \bm{r}_h) =
(K^{-1}\mu_f \bm{w}, \bm{r}_h ) - (p, \ddiv \bm{r}_h),
\quad \forall \bm{r}_h \in \bm{W}_h, \label{eqn:elliptic-proj-w}\\
& (\ddiv \bar{\bm{w}}_h, q_h) =  (\ddiv \bm{w}, q_h), \quad \forall q_h \in Q_h, \label{eqn:elliptic-proj-p}
\end{align}
Note that the above elliptic projections are decoupled;
$\bar{\bm{w}}_h$ and $\bar{p}_h$ are defined by
\eqref{eqn:elliptic-proj-w} and \eqref{eqn:elliptic-proj-p}, which is
a mixed formulation of the Poisson equation.  Therefore, the
existence and uniqueness of $\bar{\bm{w}}_h$ and $\bar{p}_h$ follow
directly from standard results.  After $\bar{p}_h$
is defined, $\bar{\bm{u}}_h$ is then determined by solving
\eqref{eqn:elliptic-proj-u}, which is a linear elasticity problem, and again
the existence and uniqueness of $\bar{\bm{u}}_h$ follow from
standard results.   Now, we split
the errors as follows,
\begin{align*}
& \bm{u}(t_n) - \bm{u}_h^n = \left( \bm{u}(t_n) - \bar{\bm{u}}_h(t_n) \right) - \left(  \bm{u}_h^n - \bar{\bm{u}}_h(t_n)  \right) =: \rho_{\bm{u}}^n - e_{\bm{u}}^n, \\
&\bm{w}(t_n) - \bm{w}_h^n = \left( \bm{w}(t_n) - \bar{\bm{w}}_h(t_n) \right) - \left(  \bm{w}_h^n - \bar{\bm{w}}_h(t_n)  \right) =: \rho_{\bm{w}}^n - e_{\bm{w}}^n, \\
&p(t_n) - p_h^n = \left( p(t_n) - \bar{p}_h(t_n) \right) - \left(  p_h^n - \bar{p}_h(t_n)  \right) =: \rho_{p}^n - e_{p}^n.
\end{align*}

\begin{lemma}
The following error estimates for the elliptic projections defined in \eqref{eqn:elliptic-proj-u}-\eqref{eqn:elliptic-proj-p} hold for $t>0$,
\begin{align}
& \| \rho_{\bm{u}} \|_1 \leq c h \left( \| \bm{u} \|_2 + \| p \|_1 \right), \label{ine:rho_u} \\
& \| \rho_{\bm{w}} \| \leq c h \| \bm{w} \|_1, \label{ine:rho_w} \\
& \| \rho_p \| \leq c h \left(  \| p \|_1 + \| \bm{w} \|_1 \right). \label{ine:rho_p}
\end{align}
\end{lemma}
\begin{proof}
Error estimates in \eqref{ine:rho_w} and \eqref{ine:rho_p} follow
from the error analysis of the mixed formulation of Poisson
problems. 
The estimate \eqref{ine:rho_u} follows from the triangle inequality: 
\[
\| \rho_{\bm{u}} \|_1 \leq \| \bm{u} - \Pi_1\bm{u} \|_1 + 
\| \Pi_1{\bm u} - \bar{\bm{u}}_h \|_1,
\]
where $\Pi_1 \bm{u}$ is the linear interpolant of ${\bm u}$.
Using the coercivity of $a^D(\cdot, \cdot)$ and that
$a^D(\Pi_1{\bm u}, {\bm v}_h) = a(\Pi_1{\bm u}, {\bm v}_h) \ \forall
{\bm v}_h\in {\bm V}_h$, for the linear function, $\Pi_1{\bm u}$, we get, 
\begin{eqnarray*}
\alpha_{\bm V}^D \|\Pi_1{\bm u} - \bar{\bm{u}}_h \|^2_1 \leq a^D(\Pi_1{\bm u} - \bar{\bm{u}}_h, \Pi_1{\bm u} - \bar{\bm{u}}_h) = a(\Pi_1{\bm u},\Pi_1{\bm u} - \bar{\bm{u}}_h) - a^D(\bar{\bm{u}}_h, \Pi_1{\bm u} - \bar{\bm{u}}_h)
\end{eqnarray*}
Taking into account that $\bar{\bm{u}}_h$ is the solution of equation \eqref{eqn:elliptic-proj-u},
\begin{eqnarray*}
\alpha_{\bm V}^D \|\Pi_1{\bm u} - \bar{\bm{u}}_h \|^2_1 &\leq& a(\Pi_1{\bm u} - {\bm{u}},\Pi_1{\bm u} - \bar{\bm{u}}_h) + \alpha(p-\bar{p}_h,\ddiv(\Pi_1{\bm u} - \bar{\bm{u}}_h))\\
&\leq& C_{\bm V}^D \|\Pi_1{\bm u} - {\bm{u}}\|_1 \|\Pi_1{\bm u} - \bar{\bm{u}}_h\|_1 + \alpha\|p-\bar{p}_h\| \|\Pi_1{\bm u} - \bar{\bm{u}}_h\|_1.
\end{eqnarray*}
Then, it follows that 
\begin{equation*}
\| \rho_{\bm{u}} \|_1 \leq \left(1 + \frac{C^D_{\bm V}}{\alpha_{\bm V}^D} \right)\| \bm{u} - \Pi_1\bm{u} \|_1 + \frac{\alpha}{\alpha_{\bm V}^D}  \| \rho_p \|
\end{equation*}
The error estimate in \eqref{ine:rho_u} is obtained by using \eqref{ine:rho_p} and the standard error estimates for linear finite elements.
\end{proof}
We similarly define the elliptic projection,
$\overline{\partial_t \bm{u}_h}$, $\overline{\partial_t \bm{w}_h}$, and
$\overline{\partial_t p_h}$ of $\partial_t \bm{u}_h$, $\partial_t \bm{w}_h$,
and $\partial_t p_h$ respectively. This gives similar estimates as
above
for $\partial_t \rho_{\bm{u}}$, $\partial_t \rho_{\bm{w}}$, and
$\partial_t \rho_p$, where on the right-hand side of the
inequalities we use norms of $\partial_t \bm{u}_h$,
$\partial_t \bm{w}_h$, and $\partial_t p_h$ instead of the norms of
$\bm{u}_h$, $\bm{w}_h$, and $p_h$ respectively.

Next, we estimate the errors, $e_{\bm{u}}$, $e_{\bm{w}}$, and
$e_{p}$ using the following norm,
\begin{equation*}
\| (\bm u_h, \bm w_h, p_h) \|_{\tau,h} := \left(  \| \bm u_h \|_1^2 + \tau \| \bm w_h \|_{K^{-1}\mu_f}^2 + \left(\frac{1}{M} + 1\right)\| p_h \|^2  \right)^{1/2},
\end{equation*}
where $\| \bm{w}_h \|_{K^{-1}\mu_f}^2:= (K^{-1}\mu_f \bm{w}_h, \bm{w}_h )$.


\begin{lemma} \label{lem:error-e}
Let $R_{\bm{u}}^j := \partial_t \bm{u}(t_j) -
\frac{\bar{\bm{u}}_h(t_j) - \bar{\bm{u}}_{h}(t_{j-1})}{\tau}$. Then,
\begin{equation} \label{ine:error-e}
\|  (e_{\bm u}^m, e_{\bm w}^m, e_p^m) \|_{\tau, h} \leq c \left(  \| e_{\bm u}^0 \|_1 + \frac{1}{M} \|e_p^0\| + \tau \sum_{j=1}^n \| R_{\bm u}^j \|_1   \right).
\end{equation}
\end{lemma}
\begin{proof}
  Choosing ${\bm v} = {\bm v}_h$ in \eqref{variational1},
  ${\bm r} = {\bm r}_h$ in \eqref{variational2}, and $q = q_h$ in
  \eqref{variational3}, subtracting these equations from
  \eqref{discrete1s_D}, \eqref{discrete2s_D} and \eqref{discrete3s_D},
  and using the definition of elliptic projections given in
  \eqref{eqn:elliptic-proj-u}, \eqref{eqn:elliptic-proj-w}, and
  \eqref{eqn:elliptic-proj-p} yields,
\begin{align}
& a^D(e_{\bm u}^m, \bm v_h) - (\alpha e_{p}^m, \ddiv \bm v_h) =  0, \label{eqn:eu}\\
& (K^{-1}\mu_f e_{\bm w}^m, \bm r_h)_h - (e_p^m, \ddiv \bm r_h) = 0,\label{eqn:ew} \\
&-\left(\frac{1}{M}\bar{\partial}_t e_p^m, q_h \right) - (\alpha \ddiv \bar{\partial}_t e_{\bm u}^m, q_h) - (\ddiv e_{\bm w}^m, q_h) = - (\ddiv R_{\bm u}^m, q_h). \label{eqn:ep}
\end{align}
Then, choosing $\bm v_h = \bar{\partial}_t e_{\bm u}^m$,
$\bm r_h = e_{\bm w}^m$ and $q_h = -e_p^m$ in \eqref{eqn:eu},
\eqref{eqn:ew}, and \eqref{eqn:ep}, respectively, and adding these equations, yields,
\begin{align*}
\| e_{\bm u}^m \|_{a^D}^2 + \tau \| e_{\bm w}^m \|_{K^{-1}\mu_f}^2 + \frac{1}{M} \|e_p^m\|^2& \leq
\| e_{\bm u}^m \|_{a^D} \| e_{\bm u}^{m-1} \|_{a^D} + \frac{1}{M}\|e_p^m\| \|e_p^{m-1}\| + \tau \| \ddiv R_{\bm u}^m \| \| e_p^m \|.
\end{align*}
Using the inf-sup condition corresponding to the mixed formulation of
the Darcy problem with RT0-P0 and using the equality in \eqref{eqn:ew} gives,
\begin{equation} \label{ine:infsup-ep}
\| e_p^m \| \leq c \sup_{0 \neq \bm r_h \in \bm W_h} \frac{(e_p^m, \ddiv \bm r_h)}{\| \bm r_h \|_{{\bm H}(\ddiv)}} = c \sup_{0 \neq \bm r_h \in \bm W_h} \frac{(K^{-1}\mu_f e_{\bm w}^m, \bm r_h)}{\| \bm r_h \|_{{\bm H}(\ddiv)}} \leq \bar{c} \| e_{\bm w}^m \|_{K^{-1}\mu_f}.
\end{equation}
By applying $ab \leq \displaystyle\frac{a^2}{2}+\frac{ b^2}{2}$ and the bound in \eqref{ine:infsup-ep}, the following inequality holds,
\begin{equation*}
\| e_{\bm u}^m \|_{a^D}^2 + \tau \| e_{\bm w}^m \|_{K^{-1}\mu_f}^2 + \frac{1}{M} \|e_p^m\|^2 \leq \| e_{\bm u}^{m-1} \|_{a^D}^2 + \frac{1}{M} \|e_p^{m-1}\|^2 + c \tau \| R_{\bm u}^m \|_1^2. 
\end{equation*}
This implies by recursion that 
\begin{align*}
\| e_{\bm u}^m \|_{a^D}^2 + \tau \| e_{\bm w}^m \|_{K^{-1}\mu_f}^2 + \frac{1}{M} \|e_p^m\|^2 \leq \| e_{\bm u}^{0} \|_{a^D}^2 + \frac{1}{M} \|e_p^{0}\|^2 + c \tau \sum_{j=1}^{m}\| R_{\bm u}^j \|_1^2.
\end{align*}
From the coercivity and continuity of the bilinear form, $a^D(\cdot, \cdot)$, the estimate in \eqref{ine:error-e} is obtained. 
\end{proof}
Finally, following the same procedures of Lemma 8 in \cite{RGHZ2016}, we have
\begin{equation}\label{ine:Ru}
\sum_{j=1}^n \| R_{\bm u}^j \|_1 \leq c \left(  \int_0^{t_n} \| \partial_{tt} \bm u  \|_1 \mathrm{d}t + \frac{1}{\tau} \int_0^{t_n} \| \partial_t \rho_{\bm u} \|_1 \mathrm{d}t \right).
\end{equation}
Thus, we derive the following error estimates.
\begin{theorem}\label{thm:error}
  Let $\bm u$, $\bm w$, and $p$ be the solutions of
  \eqref{variational1}-\eqref{variational3} and $\bm u_h^n$,
  $\bm w_h^n $, and $p_h^n$ be the solutions of
  \eqref{discrete1s_D}-\eqref{discrete3s_D}. If the following regularity assumptions hold,
\begin{align*}
&\bm u(t) \in L^{\infty}\left((0, T], \mathbf{H}_0^1(\Omega) \right) \cap L^{\infty}\left((0, T], \mathbf{H}^2(\Omega) \right), \\
& \partial_t \bm u \in L^{1}\left((0, T], \mathbf{H}^2(\Omega) \right), \ \partial_{tt} \bm u \in L^{1}\left((0, T], \mathbf{H}^1(\Omega) \right), \\
&  \bm w(t) \in L^{\infty}\left((0, T], H_0(\ddiv, \Omega) \right) \cap L^{\infty}\left((0, T], \mathbf{H}^1(\Omega) \right), \\
& p \in L^{\infty}\left((0, T], H^1(\Omega) \right), \ \partial_t p \in L^{1}\left((0, T], H^1(\Omega) \right),
\end{align*}
then, 
\begin{align} \label{ine:error}
& \|( \bm u(t_n) - \bm u_h^n, \bm w(t_n) - \bm w_h^n, p(t_n) - p_h^n ) \|_{\tau, h}  \leq c \left\{ \| e_{\bm u}^0 \|_1 + \frac{1}{M} \|e_p^0\| +  \tau \int_{0}^{t_n} \| \partial_{tt} \bm u \|_1 \mathrm{d}t  \right. \nonumber  \\
&\qquad  \left. + h \left[  \| \bm u \|_2 + \tau^{1/2} \| \bm w\|_1 + \| \bm w\|_1 + \| p \|_1 + \int_0^{t_n} \left( \| \partial_t \bm u \|_2 + \| \partial_t p \|_1  \right) \mathrm{d}t \right] \right\}.
\end{align}
\end{theorem}

\begin{proof}
  The error estimate follows directly from
  \eqref{ine:error-e}, \eqref{ine:Ru},
  \eqref{ine:rho_u}-\eqref{ine:rho_p}, and the triangle inequality.
\end{proof}
\subsection{Practical implementation}\label{sec:implementation}
Since $d_{b}(\cdot,\cdot)$ has a diagonal matrix representation, we
can eliminate the degrees of freedom corresponding to the bubble
functions in order to have the same degrees of freedom as in the 
original P1-RT0-P0 method for the three-field formulation of the
poroelasticity problem. 
After eliminating such unknowns from \eqref{diagonal_block_form}, we obtain a
$(3\times 3)$ block discrete linear system with similar blocks:
\begin{equation}\label{block_form_elim}
\widehat{\mathcal{A}}^D = 
\left(
  \begin{array}{ccc}
A_{ll}-A_{bl}^T D_{bb}^{-1}A_{bl} & 0 & G_l-A_{bl}^T D_{bb}^{-1}G_b \\
0 & \tau M_w & \tau G \\
G_l^T- G_b^T D_{bb}^{-1} A_{bl} & \tau G^T & -M_p-G_b^T D_{bb}^{-1} G_{b} 
\end{array}
\right).
\end{equation}

\section{Stabilized P1-P0 discretization for the Stokes problem}
\label{sec:stokes}

When the permeability tends to zero in the poroelasticity problem, a
Stokes-type problem is obtained. Thus, all the results
obtained above can be directly applied to Stokes'
equations. In particular, after the elimination of the bubble functions, one
obtains a finite-element pair for the Stokes' system, based on piecewise linear finite
elements for the velocity and piecewise constant functions for the
pressure. This gives a Stokes-stable finite-element method with a 
minimum number of degrees of freedom.

To illustrate this further, consider the Stokes' problem for steady flow,
\begin{eqnarray}
-\ddiv(2\nu {\boldsymbol \varepsilon}({\bm u})-p {\bm I}) = {\bm f}, \ \ \hbox{in} \ \ \Omega, \label{stokes1}\\
\ddiv {\bm u} = 0, \ \ \hbox{in} \ \ \Omega,\label{stokes2} \\
{\bm u} = {\bm 0}, \ \ \hbox{on} \ \ \Gamma, \label{stokes3}
\end{eqnarray}
where ${\bm u}$ denotes the fluid velocity, $p$ is the pressure, $\nu$ is the viscosity constant, ${\bm f}\in (L^2(\Omega))^d$ is a given forcing term acting on the fluid, and ${\boldsymbol \varepsilon}({\bm u}) = \frac{1}{2}(\nabla {\bm u} + \nabla {\bm u}^t)$. By considering ${\bm V} = {\bm H}_0^1(\Omega) = \{{\bm u}\in {\bm H}^1(\Omega) \ | \ {\bm u} = {\bm 0} \ \hbox{on} \ \Gamma\}$ and $Q=L_0^2(\Omega)=L^2(\Omega)\slash{\mathbb R}$ as the subspace of $L^2(\Omega)$ consisting of functions with zero mean value on $\Omega$, we write the weak formulation of problem \eqref{stokes1}-\eqref{stokes3} as follows
\begin{eqnarray}
  && a^S(\bm{u},\bm{v}) - (p,\ddiv \bm{v}) = (\bm f,\bm{v}),
     \quad \forall \  \bm{v}\in \bm V, \label{variational1_S}\\
  &&  (\ddiv \bm{u},q)   = 0, \quad \forall \ q \in Q,\label{variational2_S}
\end{eqnarray}
where $a^S(\bm{u},\bm{v}) = \displaystyle a^S(\bm{u},\bm{v}) =
2\nu\int_{\Omega}{\boldsymbol \varepsilon}(\bm{u}):{	\boldsymbol \varepsilon}(\bm{v})$.
As in the previous sections, we introduce the following finite-dimensional subspaces. For velocity, let ${\bm V}_h$
be the space of piecewise linear elements enriched with the normal
components of face bubble functions. For pressure, let $Q_h$ be the
subspace of piecewise constant functions. Then, the discrete
variational formulation is given by:\\ 
Find $({\bm u}_h,p_h)\in {\bm V}_h \times Q_h$ such that 
\begin{eqnarray}
  && a^S(\bm{u}_h,\bm{v}_h) - (p_h,\ddiv \bm{v}_h) = (\bm f,\bm{v}_h),
     \quad \forall \  \bm{v}_h\in {\bm V}_h, \label{discrete1_S}\\
  &&  (\ddiv \bm{u}_h,q_h)   = 0. \quad \forall \ q_h \in Q_h,\label{discrete2_S}
\end{eqnarray}
This formulation gives rise to the following block form of the fully discrete problem,
\begin{equation}\label{block_form_S}
{\cal A}_S \left(
\begin{array}{c} 
{\bm U}_b \\ 
{\bm U}_l \\ 
{\bm P} 
\end{array}
\right) = 
{\bm b}, \ \ \hbox{with} \ \ 
{\cal A}_S = \left( 
\begin{array}{ccc} 
A_{bb} & A_{bl} & G_b \\ 
A_{bl}^T & A_{ll} & G_l \\ 
G_b^T & G_l^T & 0
\end{array} 
\right),
\end{equation}
where ${\bm U}_b$, ${\bm U}_l$, and ${\bm P}$ are the unknown vectors
corresponding to the bubble component of the velocity, the linear
component of the velocity, and the pressure, respectively.. With the aim of eliminating the degrees of freedom corresponding to the bubble functions, we replace $A_{bb}$ by a spectrally-equivalent diagonal matrix $D_{bb}$, obtaining the following block form of the coefficient matrix,
\begin{equation}\label{diagonal_block_form_S}
{\cal A}_S^D = \left( 
\begin{array}{ccc} 
D_{bb} & A_{bl} & G_b \\ 
A_{bl}^T & A_{ll} & G_l \\ 
G_b^T & G_l^T & 0
\end{array} 
\right).
\end{equation}
Finally, we eliminate unknowns corresponding to the bubbles to obtain
a 2 by 2 system,
\begin{equation}\label{block_form_elim_S}
\widehat{\cal A}_S^D = 
\left(
\begin{array}{ccc}
A_{ll}-A_{bl}^T D_{bb}^{-1}A_{bl} & G_l-A_{bl}^T D_{bb}^{-1}G_b \\
G_l^T- G_b^T D_{bb}^{-1} A_{bl} & -G_b^T D_{bb}^{-1} G_{b} 
\end{array}
\right).
\end{equation}
The resulting scheme is a stabilized P1-P0 discretization of Stokes
in which stabilization terms appear in every
sub-block. Optimal order error estimates for this stabilized scheme
follow from the estimates provided in~\cite[pp.~145-149]{GR1986} for
the pair of spaces $(\bm{V}_h,Q_h)$,
$\bm{V}_h = \bm{V}_{h,1} \oplus \bm{V}_b$.

\section{Numerical Results}
\label{sec:results}

In this section we illustrate the theoretical convergence results
obtained in previous sections. We present results for both the
poroelastic problem and for Stokes' equations.

\subsection{Poroelastic problem} \label{sec:poro_test} First we
consider the test included in Section \ref{sec:difficulties}, in order
to show the corresponding results when the stabilized P1-RT0-P0 is
considered.  Table \ref{errors_difficulties_Dbb} displays the energy
norm errors for displacement and $L^2$-norm errors for pressure obtained by
applying the scheme after diagonalizing the block
corresponding to the bubble functions, $\mathcal{A}^D$ (System \eqref{diagonal_block_form}). For this test, different values
of the parameter $\kappa$ and different mesh-sizes are considered to show
that the errors are appropriately reduced independently of the
physical parameters, in contrast to the original P1-RT0-P0 scheme (Table \ref{errors_difficulties}).
\begin{table}[htb!]
\begin{center}
\begin{tabular}{|c|c|c|c|c|c|c|}
\cline{3-7}
\multicolumn{2}{c|}{} & 
$N =8$ & $N =16$ & $N =32$ & $N =64$  & $N = 128$ \\
\hline
\multirow{2}{*}{$\kappa = 10^{-4}$} & $\|\Pi_1{\bm u}-{\bm u}_h\|_A$ & 
$0.0126$ & $0.0029$ & $0.0007$ & $0.0002$ & $4.70\times 10^{-5}$ \\
& $\|\Pi_0p-p_h\|_{L^2}$ & 
$0.0308$ & $0.0064$ & $0.0012$ & $0.0003$ & $6\times 10^{-5}$ \\
\hline
\multirow{2}{*}{$\kappa = 10^{-6}$} & $\|\Pi_1{\bm u}-{\bm u}_h\|_A$ & 
$0.0174$ & $0.0055$ & $0.0013$ & $0.0003$ & $7.04\times 10^{-5}$ \\
& $\|\Pi_0p-p_h\|_{L^2}$ & 
$0.0639$ & $0.0359$ & $0.0151$ & $0.0043$ & $7.93\times 10^{-4}$ \\
\hline
\multirow{2}{*}{$\kappa = 10^{-8}$} & $\|\Pi_1{\bm u}-{\bm u}_h\|_A$ & 
$0.0176$ & $0.0057$ & $0.0015$ & $0.0004$ & $1.05\times 10^{-4}$ \\ 
& $\|\Pi_0p-p_h\|_{L^2}$ & 
$0.0622$ & $0.0379$ & $0.0196$ & $0.0097$ & $0.0046$ \\
\hline
\multirow{2}{*}{$\kappa = 10^{-10}$} & $\|\Pi_1{\bm u}-{\bm u}_h\|_A$ & 
$0.0176$ & $0.0057$ & $0.0015$ & $0.0004$ & $1.08\times 10^{-4}$ \\ 
& $\|\Pi_0p-p_h\|_{L^2}$ & 
$0.0621$ & $0.0378$ & $0.0197$ & $0.0098$ & $0.0049$ \\
\hline
\end{tabular}
\caption{Energy norm for displacement errors and $L^2$-norm for
  pressure errors by considering different values of $\kappa$ and
  different mesh-sizes, using the ``diagonal'' bubble formulation, $\mathcal{A}^D$ \eqref{diagonal_block_form}.} 
\label{errors_difficulties_Dbb}
\end{center}
\end{table}

We also compare the obtained errors with those provided by the
fully enriched element, $\mathcal{A}$ (System \eqref{block_form}),
in order to see that the same error
reduction is achieved. Figure~\ref{figure_error_comparison}, 
displays a comparison of the displacement and pressure errors in the
energy and $L^2$ norms, respectively, for different grid sizes.  We choose $\kappa =10^{-8}$ here, though similar pictures are
obtained for different values of $\kappa$. We observe that
the slopes corresponding to both schemes are the same, although the
scheme corresponding to the diagonal version provides slightly worse
errors. However, this scheme, when the bubble block is eliminated, uses less degrees of freedom and is easily
implemented from an already existing P1-RT0-P0 code. 
\begin{figure}[htb]
\begin{center}
\begin{tabular}{cc}
\includegraphics[width = 0.4\textwidth ]{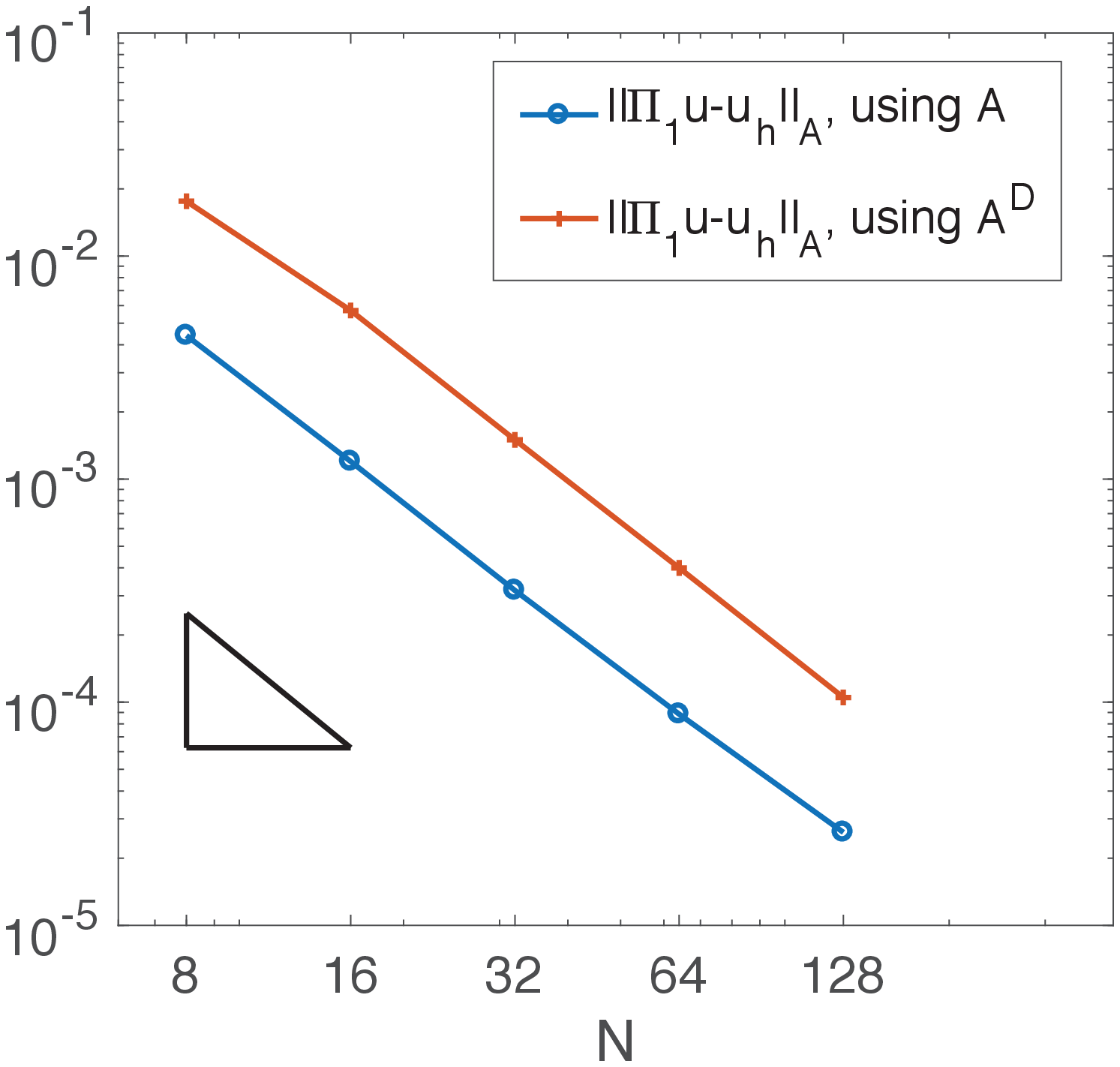}
&
\includegraphics[width = 0.42\textwidth ]{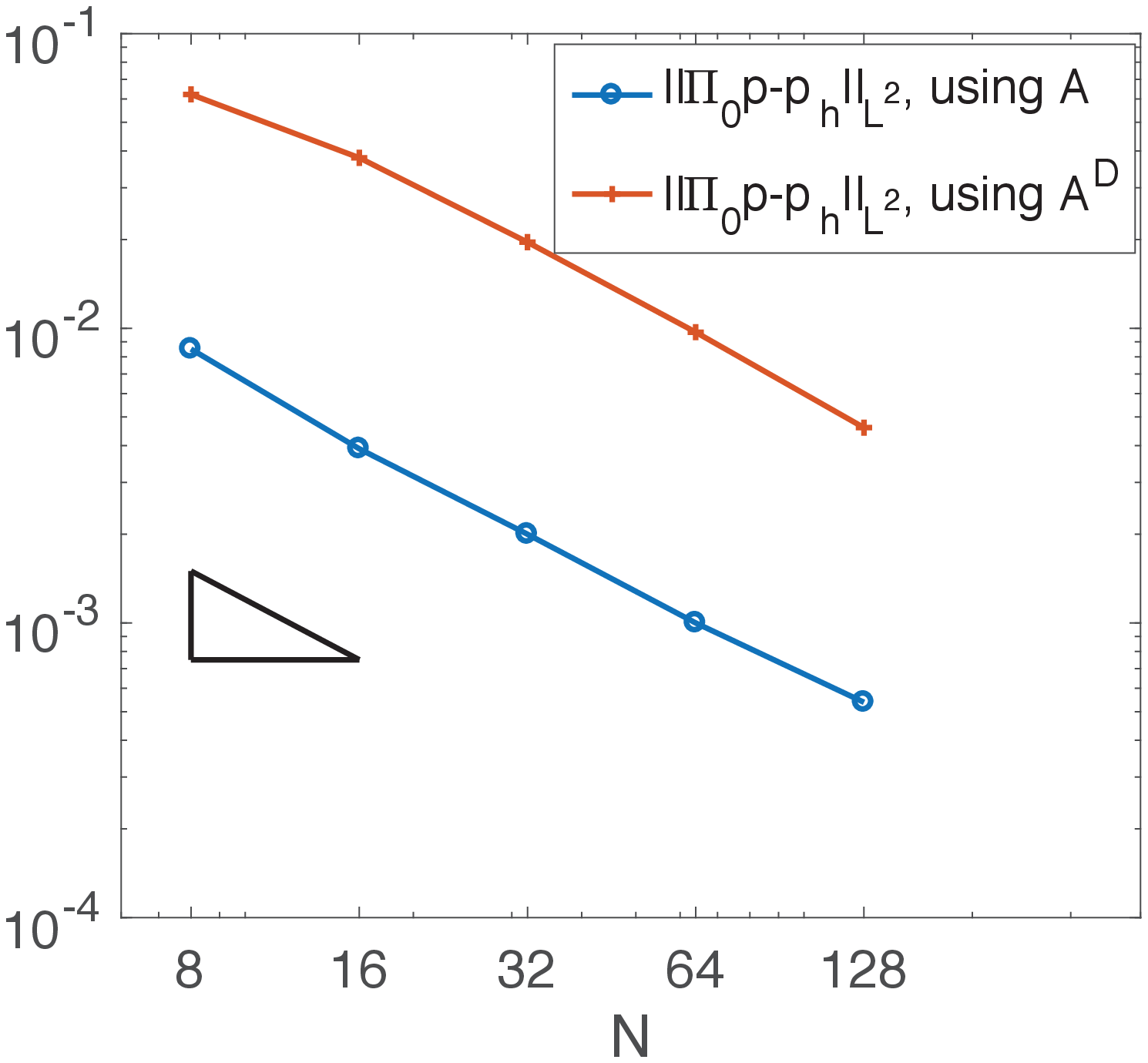}\\
(a) & (b)
\end{tabular}
\caption{Reduction of the (a) displacement and (b) pressure errors for
  different mesh-sizes, by using  the enriched finite element scheme,
  $\mathcal{A}$ \eqref{block_form}, as well as the scheme with
  diagonal block used for the bubble functions, $\mathcal{A}^D$ \eqref{diagonal_block_form}.}
\label{figure_error_comparison}
\end{center}
\end{figure}


\subsection{Stokes' problem} \label{sec:stokes_test}

While it is well-known that the P1-P0 finite element pair is not stable for
Stokes' equations, we show here that the new formulation, $\widehat{\mathcal{A}}^D_S$
\eqref{block_form_elim_S}, resulting from the elimination of the normal
components of the bubbles, does provides a stable method. Consider \eqref{stokes1}-\eqref{stokes3} on a unit square
$(0,1)\times (0,1)$, where the right-hand side ${\bm f}$ is chosen such that
the analytical solution is given by
$${\bm u}(x,y) = \left(\sin(\pi x)\cos(\pi y), -\cos(\pi x)\sin(\pi y)\right), \quad p(x,y) = 0.5 - x.$$
Figure \ref{errors_Stokes} compares the error reduction for both the
velocity and pressure using the bubble function enhanced schemes
described by $\mathcal{A}_S$ \eqref{block_form_S} and
$\mathcal{A}_S^D$ \eqref{diagonal_block_form_S}.
\begin{figure}[htb]
\begin{center}
\begin{tabular}{cc}
\includegraphics[height = 5.95cm ]{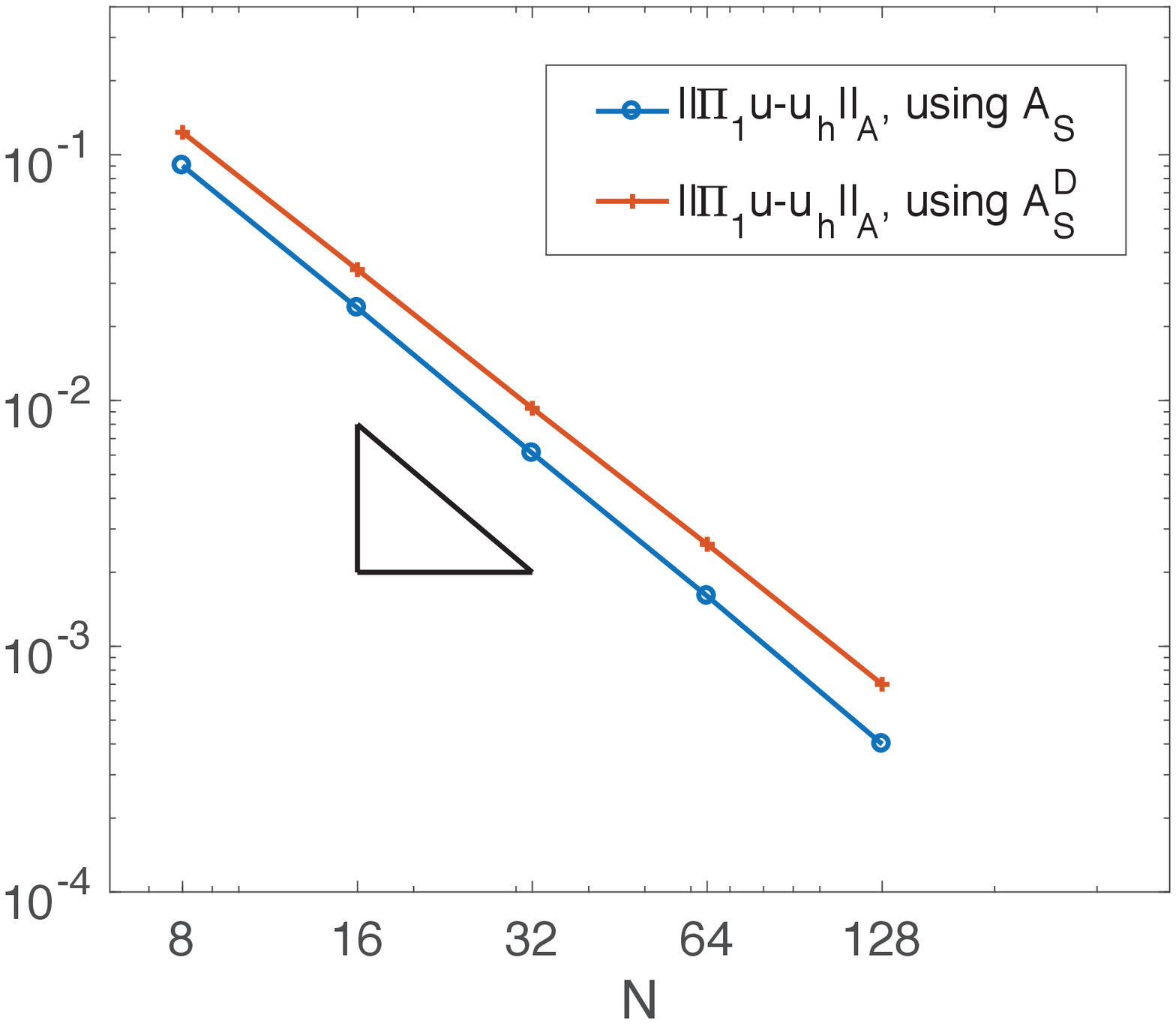}
&
\includegraphics[height = 5.9cm ]{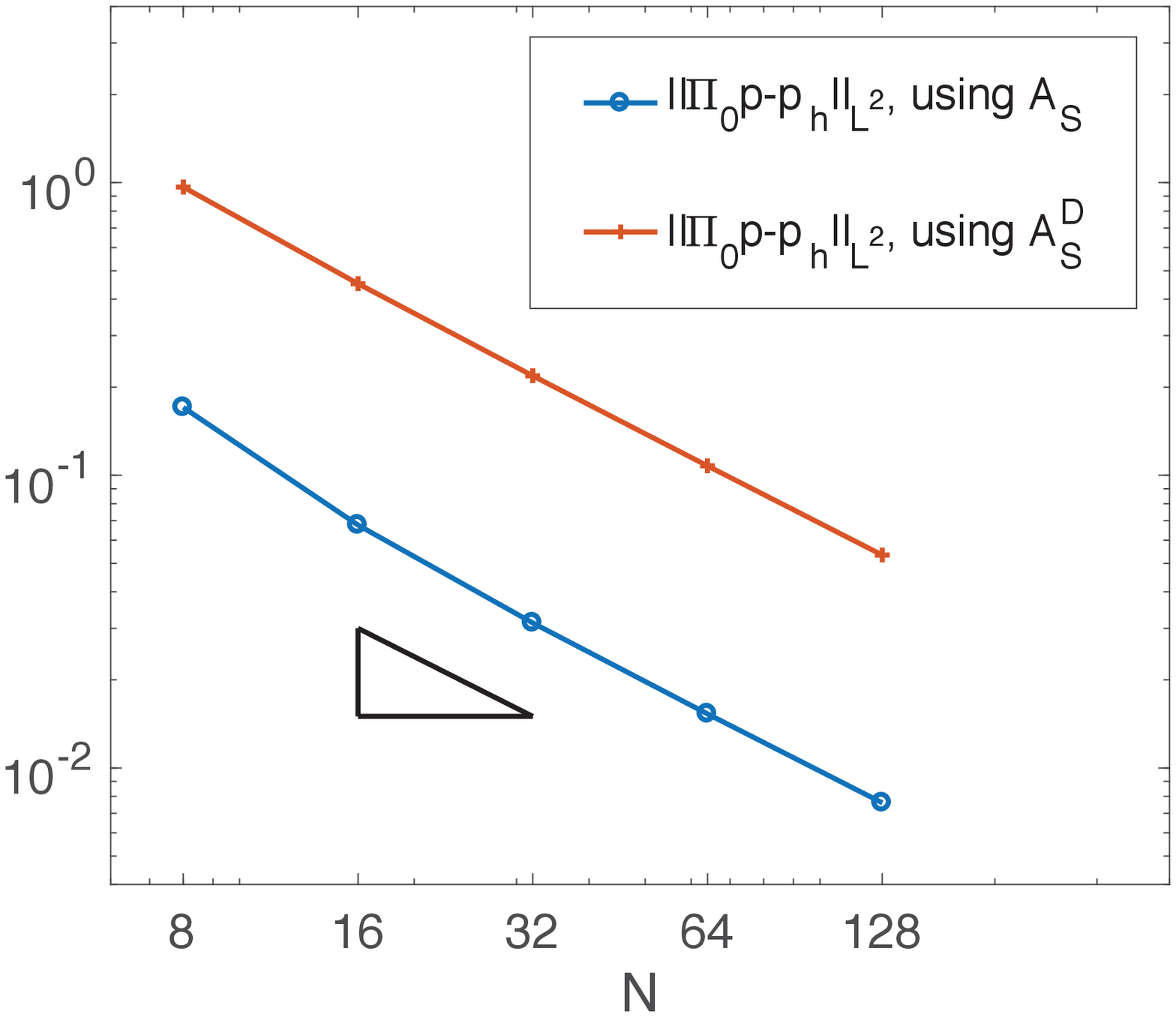}\\
(a) & (b)
\end{tabular}
\caption{Reduction of the (a) velocity and (b) pressure errors for
  different mesh-sizes, by using  the enriched finite element scheme,
  $\mathcal{A}_S$ \eqref{block_form_S}, as well as the scheme
  resulting from using a diagonal block for the bubble functions, $\mathcal{A}^D_S$ \eqref{diagonal_block_form_S}.}
\label{errors_Stokes}
\end{center}
\end{figure}
Here, for the scheme~$\mathcal{A}_S$ \eqref{block_form_S}, the energy
norm for the velocity is defined as $\| \bm{v} \|_A^2 := a^S(\bm{v},
\bm{v}) = 2\mu ({\boldsymbol \varepsilon}(\bm{v}), {\boldsymbol
  \varepsilon}(\bm{v})) $ for $\bm{v} \in \bm{V}$ .  In addition, for
the scheme~$\widehat{\mathcal{A}}_S^D$ \eqref{diagonal_block_form_S}, the energy norm is defined as $\| \bm{v} \|_A^2 := a^{S,D}(\bm{v}, \bm{v})$ where $a^{S,D}(\bm{u}, \bm{v})$ is defined as in~\eqref{ad-form} with $a(\cdot, \cdot)$ replaced by $a^S(\cdot, \cdot)$.  Both methods give the same, optimal, order of convergence,
demonstrating that the
inclusion of the bubble functions guarantee the stability of the
method.  Moreover, though the errors are slightly
higher, the elimination of the bubble functions would provide a
stable convergent method, but reduces the problem to one that contains the
same number of degrees of freedom as the P1-P0 discretization.  Thus,
we get a stable scheme with no increase in cost.  

\section{Conclusions}
\label{sec:conclusions}
In this paper, we have shown how to stabilize the popular P1-RT0-P0
finite-element discretization for a three-field formulation of the
poroelasticity problem. By adding the normal
components of the bubble basis functions associated with the faces of
the triangulation to the P1 element for displacements, we have
demonstrated that an inf-sup condition is satisfied independently of the
physical and discretization parameters of the problem. Moreover, the degrees of freedom
added to the faces are eliminated resulting in a stable scheme with
the same number of unknowns as in the initial P1-RT0-P0
discretization. Furthermore, this idea has been extended to the
Stokes' equations, yielding a
\textit{stable finite-element formulation} with the lowest
possible number of degrees of freedom, \textit{equivalent to a P1-P0
discretization}.  Future work includes investigating such formulations
and their performance for various applications in poroelasticity, and
extending the discretization to other PDE systems which have simliar
properties to the Stokes' equations.

\section*{Acknowledgements}

The work of F.~J.~Gaspar is supported by the European Union's Horizon
2020 research and innovation programme under the Marie
Sklodowska-Curie grant agreement NO 705402, POROSOS. The research of
C.~Rodrigo is supported in part by the Spanish project FEDER /MCYT
MTM2016-75139-R and the DGA (Grupo consolidado PDIE). The work of
Zikatanov was partially supported by NSF grants DMS-1418843 and
DMS-1522615.  The work of Adler and Hu was partially supported by NSF
grant DMS-1620063.

\bibliographystyle{elsarticle-num}
\bibliography{bib_Poro_P1b_RT0_P0}

\end{document}